\documentclass[a4paper]{amsart}
\usepackage[margin=3.5cm]{geometry}
\usepackage[utf8]{inputenc}
\usepackage{amsmath,amsfonts,amssymb,amsthm, multirow}
\usepackage{comment, stackrel}
\usepackage{hyperref}
\usepackage{makecell}
\usepackage{hhline}
\usepackage{tikz}
\usepackage{tikz-cd}

\theoremstyle{plain}
\newtheorem{theorem}{Theorem}[section]

\newtheorem{proposition}[theorem]{Proposition}

\renewcommand{\phi}{\varphi}
\renewcommand{\tilde}{\widetilde}

\theoremstyle{definition}

\theoremstyle{remark}
\newtheorem{remark}[theorem]{Remark}

\newcommand{\Z}{\mathbb{Z}}                    
\newcommand{\Q}{\mathbb{Q}}                    
\newcommand{\R}{\mathbb{R}}                    
\newcommand{\C}{\mathbb{C}}                    
\renewcommand{\P}{\mathbb{P}}                    %
\newcommand{\A}{\mathbb{A}}                    %
\newcommand{\F}{\mathbb{F}}                    
\renewcommand{\O}{\mathcal{O}}                    %
\newcommand{\M}{\mathcal{M}}                    %
\newcommand{\CF}{\mathcal{F}} 
\newcommand{\CQ}{\mathcal{Q}} 
\newcommand{\VV}{\mathbb{V}}

\newcommand{\q}{{\bf q}}

\newcommand{\Cl}{\operatorname{Cl}}

\newcommand{\Proj}{\operatorname{Proj}}

\newcommand{\Hom}{\operatorname{Hom}}

\newcommand{\im}{\operatorname{im}}
\newcommand{\Bs}{\rm Bs\,}
\newcommand{\Sing}{{\rm Sing}}

\title{On K3 fibred Calabi--Yau threefolds in weighted scrolls}

\author{Geoffrey Mboya}
\address{Mathematical Institute, 
University of Oxford, 
United Kingdom}
\email{mboya@maths.ox.ac.uk}

\author{Bal\'{a}zs Szendr\H{o}i}
\address{Faculty of Mathematics, 
University of Vienna, 
Austria}
\email{balazs.szendroi@univie.ac.at}

\begin{document}
\begin{abstract}
The aim of this paper is to classify mildly singular Calabi--Yau threefolds fibred in low-degree weighted K3 surfaces and embedded as anticanonical hypersurfaces in weighted scrolls, extending results of Mullet. We also study projective degenerations, revisiting an example due to Gross and Ruan. Finally we briefly discuss the general question of embedding a projective fibration into a weighted scroll. 
\end{abstract}

\maketitle

\section{Introduction}

Let $(X,D)$ be a polarized projective variety over $\C$: a pair consisting of a normal projective $\C$-variety $X$ and an ample $\Q$-Cartier Weil divisor $D$ on $X$, with associated divisorial sheaf $\O_X(D)$. We can then consider the graded $\C$-algebra
\[ R = \bigoplus_{m=0}^\infty H^0(X, \O_X(mD)).\]
Ampleness of $D$ is equivalent to an isomorphism $X\cong\Proj R$; in particular, the algebra $R$ is finitely generated. 
Choosing a set of algebra generators $r_1, \ldots, r_n\in R$ of positive integer weights $c_1, \ldots, c_n$ gives a surjection 
\[ S=\C[x_1, \ldots, x_n] \twoheadrightarrow R\] of graded algebras, and a corresponding embedding
\begin{equation} (X,D) \hookrightarrow \left(\P^{n-1}[c_1, \ldots, c_n], \O(1)\right) 
\label{emb}\end{equation}
into a weighted projective space. Under suitable vanishing assumptions, the dimension of the $m$-th graded piece of $R$ is given by a Riemann--Roch type formula, allowing the
computation of the Hilbert series of $R$.  

Miles Reid's {\em graded ring method} studies classes of projective varieties of increasing complexity according to their codimension in the embedding~\eqref{emb}, using the information derived from the Hilbert series. Well-studied classical examples include the ``famous 95'' K3 hypersurfaces in weighted projective 3-space due to Iano-Fletcher and Reid~\cite[Sect.13.3]{fletcher} and the attached families of $\Q$-Fano 3-folds~\cite[Sect.16.6]{fletcher}, \cite{ABR}; the codimension 2 complete intersection K3 surfaces of Iano-Fletcher~\cite[Sect.13.8]{fletcher}; the codimension 3 Pfaffian K3 families~\cite{altinok} and others. The current state of the art is contained in the Graded Ring Database~\cite{grd, gk}. 

In a paper~\cite{mullet} by Mullet, the first steps were taken to study a relative version of these constructions: the case of a polarized projective fibration $f\colon X\rightarrow B$ over a polarized base~$B$. In that paper, the following specific setup was studied: 
\begin{enumerate}
\item[(i)] the base $B\cong\P^1$; 
\item[(ii)] the general fibre of $f$ is one of the ``famous 95'' list of Iano-Fletcher--Reid;
\item[(iii)] the fibration $f\colon X\rightarrow\P^1$ embeds into a weighted scroll $\pi\colon \F\rightarrow\P^1$ as a quasi-smooth anticanonical hypersurface.
\end{enumerate}

Our aim is to study this problem further, extending the study of K3 hypersurface-fibred Calabi--Yau threefolds in two ways. On the one hand, we return to the first few cases of the Iano-Fletcher--Reid list, and for these fibres relax assumption (iii),
to allow for isolated singularities along the base locus. This allows for a longer list of examples, with some new cases of potential further interest; 
all our examples have projective Calabi--Yau resolutions. 
On the other hand, we use the explicit realisation of some of the threefolds as hypersurfaces in scrolls to study their projective geometry including degenerations, 
re-visiting from a different point of view an example due to Gross~\cite{gross}, also studied by Ruan~\cite{ruan}. 

We begin in Section~\ref{sec:defs} by recalling the definition and basic properties of weighted scrolls, and discuss their toric geometry and anticanonical hypersurfaces. Our main results are contained in Section~\ref{sec:K3fam}. In Theorems~\ref{thm_quartic} and \ref{thm_others}, we classify K3-fibred Calabi--Yau threefolds with mild singularities embedded in weighted scrolls over $\P^1$, whose generic fibres are quartic, quintic or sextic weighted K3's. We also study aspects of the projective geometry of some of our examples. 

Other searches of a similar nature may also be of interest; the first author's thesis~\cite{mboya} will study, and list, polarized elliptic K3 surfaces with rational double point singularities, as 
well as higher codimension examples. 

In the closing Section~\ref{sect_end}, we speculate on how far the conditions (i)--(iii) can be relaxed to find interesting examples of projective fibrations, leaving detailed investigations of these questions to future work. 

\vspace{0.1in}

\noindent {\bf Acknowledgements} The first-named author would like to express his thanks to the Mathematics and Physical Sciences programme of the Simons Foundation (Award ID: 599410) and the Mathematical Institute, University of Oxford for support during his doctoral studies.

\section{Preliminaries on weighted scrolls}
\label{sec:defs}

\subsection{Basics}

We write this section in the generality needed for our examples, in order not to over-burden notation, following the treatment in~\cite{reid_chapters} in the unweighted case. Fix an integer $n>1$, a set of positive integer weights $(b_1, b_2, \ldots, b_n)$ and arbitrary integer twists $(a_1, \ldots, a_n)$. To simplify the situation and because this is the case in all our applications, we assume from the outset that the first weight $b_1=1$; as usual, the remaining weights will be ordered in weakly increasing order. 

Consider the action 
\begin{equation}
(\lambda, \mu) \colon (t_1, t_2; x_1,\ldots, x_n) \mapsto (\lambda t_1, \lambda t_2; \lambda^{-a_1}\mu^{b_1}x_1,\ldots, \lambda^{-a_n}\mu^{b_n}x_n)
\label{action}
\end{equation}
of $(\C^*)^2$ on affine space $\C^2\times \C^n$. This action preserves the open subset
\[ U = (\C^2\setminus \{0\}) \times (\C^n\setminus \{0\}).
\]
We define the weighted scroll of type $(a_1,\ldots, a_n |\, b_1, \ldots, b_n)$ to be the quotient
\[ \F(a_1,\ldots, a_n |\, b_1, \ldots, b_n) = U/(\C^*)^2.
\] 
with quotient map
\[ q\colon U \to \F(a_1,\ldots, a_n |\, b_1, \ldots, b_n).
\]
The $(\C^*)^2$-orbits are easily checked to be closed on $U$, and it is known that $q$ is a geometric quotient. The $(\C^*)^2$-action has finite stabilizers along certain loci in $U$, leading to finite
quotient singularities on $\F(a_1,\ldots, a_n |\, b_1, \ldots, b_n)$. There are no stabilisers if $b_2=\ldots =b_n=1$ also. 

It is clear that first projection on $U\subset\C^2\times \C^n$ is compatible with taking a quotient, so we get a morphism 
\[ \pi\colon \F(a_1,\ldots, a_n |\, b_1, \ldots, b_n) \to \P^1
\] 
with fibres that are all isomorphic to the weighted projective space $\P^{n-1}[b_1, b_2, \ldots. b_n]$. 

Note that given weights $(1, b_2, \ldots, b_n)$, replacing a set of twists $(a_1, \ldots, a_n)$ by the set $(a_1+k, a_2+kb_2\ldots, a_n+kb_n)$ leads to an isomorphic quotient. Thus, assuming 
$b_1=1$ as we do, we may also assume $a_1=0$.  

Algebraically, we can considering the coordinate algebra
\begin{equation}\label{ring S}
 S = \C[t_1, t_2, x_1,\ldots, x_n]
\end{equation}
of $\C^2\times \C^n$. The $(\C^*)^2$-action translates into the bigrading on $S$ that assigns degree $(1,0)$ to the generators $t_1, t_2$ and $(-a_i, b_i)$ to the generator $x_i$. 
We can, and will, use the 
variables $(t_i; x_j)$ analogously to projective coordinates on (weighted) projective spaces. 

\begin{remark} Fix a further integer $k>1$, and another set of positive integer weights $(c_1, \ldots c_k)$. The action of $(\C^*)^2$ on affine space $\C^k\times \C^n$ defined by
\[
(\lambda, \mu) \colon (t_1, \ldots, t_k; x_1,\ldots, x_n) \mapsto (\lambda^{c_1} t_1, \ldots, \lambda^{c_k} t_k; \lambda^{-a_1}\mu^{b_1}x_1,\ldots, \lambda^{-a_n}\mu^{b_n}x_n)
\]
gives rise to a weighted scroll $\F(c_1, \ldots c_k || a_1,\ldots, a_n |\, b_1, \ldots, b_n)=U/(\C^*)^2$ with a map \[\pi\colon\F(c_1, \ldots c_k || a_1,\ldots, a_n |\, b_1, \ldots, b_n)\rightarrow \P^{k-1}[c_1, \ldots, c_k],\] whose fibres are weighted projective spaces  $\P^{n-1}[b_1, b_2, \ldots. b_n]$. 
We do not use this generality in the main body of the paper, but will return to this more general construction at the end. 
\label{general}
\end{remark} 

\subsection{The toric description}\label{sect_toric}
Weighted scrolls are clearly toric varieties; in this section, we make this explicit. For simplicity, in this section $\F$ denotes a scroll $\F(a_1,\ldots, a_n |\, b_1, \ldots, b_n)$ over
$\P^1$ given by twists and weights as above, with $b_1=1$. 

Start with the exact sequence of abelian groups 
\begin{equation} 0 \longrightarrow \Z^2 \longrightarrow \Z^{2+n} \longrightarrow N \longrightarrow 0, \label{defN} \end{equation}
where the first map maps the standard generators of $\Z^2$ to the vectors $(1,1,-a_1, \ldots, -a_n)$ and $(0,0,1, b_2, \ldots, b_n)$ respectively. Note that by assuming $b_1=1$, we
automatically get that~$N$ is a free $\Z$-module. The standard lattice generators of $\Z^{2+n}$ project to elements of~$N$ that we denote by $\sigma_1, \sigma_2, \rho_1, \ldots, \rho_n$. Note that by our definition of~$N$, they satisfy the relations
\[ \sigma_1+\sigma_2 - \sum_{j=1}^n a_j \rho_j=0, \ \ \ \ \sum_{j=1}^n b_j \rho_j = 0.
\]
So with $b_1=1$, we see that $\{\sigma_2, \rho_2, \ldots, \rho_n\}$ is a $\Z$-basis of~$N$. 

We define a fan $\Sigma$ in the $n$-dimensional real space $N_\R$ by declaring the maximal dimensional cones to be
\[\tau_{i,j} = {\rm Span}(\sigma_i, \rho_1, \ldots, \rho_{j-1}, \rho_{j+1} \ldots, \rho_{n});\]
it is easy to check that this is indeed a fan.

\begin{proposition} Using standard toric notation, the toric variety $X_{N,\Sigma}$ is isomorphic to the scroll $\F=\F(a_1,\ldots, a_n |\, b_1, \ldots, b_n)$ defined above, whereas the bigraded algebra 
$S$ in~\eqref{ring S} becomes the Cox ring~\cite{cox} of this toric variety. 
\end{proposition} 
\begin{proof} Tensoring~\eqref{defN} by $\C^*$, we get an exact sequence of tori
\[ 1 \longrightarrow (\C^*)^2 \longrightarrow  (\C^*)^{2+n} \longrightarrow T \longrightarrow 1\]
mapping $(\lambda, \mu)$ to $(\lambda, \lambda, \lambda^{-a_1}\mu^{b_1},\ldots, \lambda^{-a_n}\mu^{b_n})$. 
This precisely describes the action of $(\C^*)^2$ on $\C^2\times \C^n$ from~\eqref{action} as a subgroup of the standard torus of $\C^2\times \C^n$. Now the description of~$\F$ above can be directly compared to the standard toric description of $X_{N,\Sigma}$ from~\cite{fulton}, to check that everything agrees. In particular, the irrelevant ideal~\cite{cox} corresponding to our chosen fan $\Sigma$ is the ideal 
\[ B=\langle t_ix_j : 1\leq i\leq 2, 1\leq j\leq n\rangle\lhd S,\]
the vanishing locus $Z\subset \C^2\times \C^n$ of which is the complement of our open set $U$. 
\end{proof}

Let $\Cl(\F)$ denote the group of Weil divisors of $\F$. The one-dimensional cones in our fan give us divisors $D_{\sigma_i}=\{t_i=0\}$ and $D_{\rho_j} = \{x_j=0\}$ on $\F$. Let $L=[D_{\sigma_1}]\in\Cl(\F)$ and
$M=[D_{\rho_1}]\in\Cl(\F)$. 

Assume that, using $b_1=1$, we have adjusted the twist so that $a_1=0$. 
\begin{proposition} We have \[ \Cl(\F) \cong \Z L \oplus \Z M.\] 
For $n,m\in Z$, we have 
\[ H^0(\F, nL + mM) \cong S_{n,m},
\]
the bidegree $(n,m)$ piece of the graded algebra $S$. 
The canonical class of $\F$ is given by 
\begin{equation} 
\label{canonical class}
K_\F = \left(-2 + \sum_{j=2}^n a_j \right)L - \left( -1 -\sum_{j=2}^n b_j\right)M\in\Cl(\F).
\end{equation}
Finally, $\F$ is $\Q$-factorial: every divisor on it is $\Q$-Cartier. 
\end{proposition} 
\begin{proof} This is all well known toric geometry. With $M=\Hom(N,\Z)$, the dual of the sequence~\eqref{defN} is the sequence
\[ 0 \longrightarrow M \longrightarrow \Z^{2+n} \longrightarrow \Cl(\F) \longrightarrow 0, \]
with the first map given by $m\mapsto \sum_\rho \langle m,\rho\rangle D_\rho$, where the sum runs over all 1-dimensional cones in the fan $\Sigma$, and angle brackets denote the canonical pairing between $M$ and $N$. So the class group of $\F$ is free of rank 2. Also elements of $M$ give relations between the classes of various torus-invariant divisors in $\Cl(\F)$; explicitly these become the relations
\[ [D_{\sigma_1}]-[D_{\sigma_2}]=0 \mbox{ and for all } 2\leq j \leq n, \ a_j [D_{\sigma_1}]  - b_j [D_{\rho_1}] + [D_{\rho_j}] = 0. \]
We see that indeed $L=[D_{\sigma_1}]$ and $M=[D_{\rho_1}]$ give a $\Z$-basis for $\Cl(\F)$. The next statement is a standard part of the Cox ring package. The canonical class of $\F$ is 
\[ K_\F = -[D_{\sigma_1}] -[D_{\sigma_2}] - \sum_{j=1}^n [D_{\rho_j}] =  \left(-2 + \sum_{j=2}^n a_j \right)L - \left( -1 -\sum_{j=2}^n b_j\right)M
\]
using the relations above. Finally, as all cones in $\Sigma$ are simplicial, $\F=X_{N,\Sigma}$ is $\Q$-factorial. 
\end{proof}

\subsection{Generalities on anticanonical sections}

Fix a weighted scroll  \[ \pi\colon\F=\F(a_1,\ldots, a_n |\, b_1, \ldots, b_n)\to \P^1.\] 
Assume that the anticanonical system $|-K_{\F}|$ is nonempty. Throughout this paper, we are interested in general anticanonical hypersurfaces $X\in |-K_{\F}|$, which 
are themselves be fibered over $\P^1$. To fix notation, let \[d= 2-\sum_{j=1}^n a_j, \ \ \ e = \sum_{j=1}^n b_j\] be the negatives of the integer constants appearing in~\eqref{canonical class} above, so that 
\[-K_{\F} = dL + eM\in \Cl(\F).\] The hypersurface $X$ is then defined inside $\F$ as the zero locus of a general bihomogeneous section $f(t_i, x_j)\in S_{d,e}$. 

We have to be more explicit about the equation $f$. Recall that $S$ is bigraded, with $t_i$, $x_j$ having degree $(1,0)$ and $(-a_j, b_j)$ respectively. 
For an $n$-tuple of non-negative integers $\q=(q_1, q_2, \ldots, q_n)$, write $\q\vdash_w e$ to mean that $\q$ is 
a $(b_1, \ldots, b_n)$-weighted partition of $e$, in other words 
\[ \sum_{j=1}^n b_jq_j = e.
\]
It is then easy to see that the polynomial $f$ must have the form
\[ f(t_i, x_j) = \sum_{\q\vdash_w e} \alpha_{\q} (t_i) \prod_{j=1}^n x_j^{q_j}, 
\]
where $\alpha_{\q} (t_i)$ is a homogeneous polynomial of the variables $(t_1, t_2)$ of degree 
\[ \deg \alpha_{\q}  = d+\sum_{j=1}^n a_jq_j = 2 + \sum_{j=1}^n (q_j-1)a_j,
\]
as long as this expression is non-negative; otherwise of course the monomial $\prod_{j=1}^n x_j^{q_j}$ does not appear in~$f$.

Recall the quotient construction $\F= U/(\C^*)^2$ together with the quotient map $q\colon U\to\F$, and let $\tilde X=q^{-1}(X) \subset U$, the zero-locus of the polynomial $f$ inside the set $U\subset\A^{2+n}$.  Let $Z\subset\F$ be a $T$-invariant closed subvariety. We call a hypersurface $X\subset\F$ {\em quasismooth away from~$Z$}, if $q^{-1}(X\setminus Z) \subset q^{-1}(\F\setminus Z)$ 
is nonsingular. In this case, the only singularities of $X$ away from $Z$ are finite quotient singularities arising from finite stabiliser subgroups inside~$(\C^*)^2$. 

We recall also that $X\subset\F$ is {\em well-formed}, if the codimension of $\Sing(\F)\cap X\subset X$ is at least two.

Denote $B=\Bs|-K_\F|$ the base locus of the anticanonical system of $\F$.

\begin{proposition}\begin{enumerate}\label{toric lemma}
\item[(i)] The subset $B\subset\F$ is $T$-invariant. 
\item[(ii)] The general anticanonical hypersurface $X\subset\F$ is quasismooth away from $B$. 
\item[(iii)] If $X$ is well-formed, and has Gorenstein singularities along $B$, then it is a Calabi--Yau threefold. 
\end{enumerate}
\end{proposition}
\begin{proof} (i) is in fact true more generally for any divisor class, and follows from the fact that the torus $T$ acts trivially on the discrete group $\Cl(\F)$. 
(ii) is [Mullet, Prop.~6.7]. (iii) is a mild generalization of \cite[Thm.8.3]{mullet} and the proof there applies verbatim.  
\end{proof}

Thus as long as the singularities of $X$ are mild, it is a Calabi-Yau variety fibred over $\P^1$, with fibres $X_t\subset \P^{n-1}[b_1, \ldots, b_n]$ for $t\in\P^1$ that are themselves Calabi--Yau hypersurfaces. 
In order to find not-too-singular hypersurfaces $X\subset\F$, we thus need to focus on the possible singularities of $q^{-1}(B)\cap \tilde X$. 

\section{Mildly singular threefold families of K3 surfaces}
\label{sec:K3fam}

\subsection{Families of quartic K3s}
\label{subsec:quartic}

The easiest example of a K3 surface, also the first entry in the Iano-Fletcher--Reid list, is the quartic surface $S_4\subset\P^3$. The simplest example of a family of anticanonical Calabi--Yau hypersurfaces fibred in K3 surfaces is the anticanonical family $X\subset\P^1\times\P^3$. Mullet finds eight further scrolls containing nonsingular Calabi--Yau 
hypersurfaces, which are listed in Table 1 below as familes 1-9. In this case, we find one further example; note that one member of this family was discussed from a different point of view
in~\cite[Rem.2.5]{AL}.

\begin{theorem} There exists exactly one family of quartic $K3$-fibred anticanonical hypersurfaces $X\in |-K_{\F}|$ with isolated singularities along the base locus $B=\Bs |-K_{\F}|$ and quasi-smooth outside $B$: the family of anticanonical hypersurfaces \[X\subset \F(0,0,1,2).\] A general variety $X$ in this family has $3$ threefold ordinary double points along~$B\cong\P^1\times\P^1$, and is nonsingular elsewhere. Blowing up $B$ in $X$ gives a small projective Calabi--Yau resolution $Y\rightarrow X$. 
\label{thm_quartic}
\end{theorem}

The complete list of quartic families with at worst isolated singularities is given in Table~1. The last column of the table contains the dimension of $\M_{X\subset\F}$, the space of 
embedded deformations of the anticanonical hypersurface $X$ in the respective scroll $\F$; this number is easy to calculate as the difference between the number of 
parameters in $\F$ and the dimension of its automorphism group. Note that this is only a lower bound for the dimension of the space of all deformations of $X$, which is the
Hodge number $h^{2,1}$ of (a smooth model of) $X$. 

\begin{table}[h]
\centering
\begin{tabular}{|c|c|c|c| }\hline
&Scroll $\F$ & \makecell{Description of $\Bs|-K_{\F}|$ \\ and general $\text{}X\in|-K_{\F}|$}&  $\dim \M_{X\subset\F}\ $ \\ \hhline{|=|=|=|=|}
1&$\F(0,0,0,0)$&$\Bs|-K_{\F}|=\emptyset,~\text{general } X \text{ nonsingular}$ &86\\ \hline
2&$\F(0,0,0,1)$&$\Bs|-K_{\F}|=\emptyset,~\text{general } X \text{ nonsingular}$ &118\\ \hline
3&$\F(0,0,0,2)$&$\Bs|-K_{\F}|=\emptyset,~\text{general } X \text{ nonsingular}$&83\\  \hline
4&$\F(0,0,1,1)$&$\Bs|-K_{\F}|=\emptyset,~\text{general } X \text{ nonsingular}$&86\\  \hline
5&$\F(0,1,1,1)$& $\dim\Bs|-K_{\F}|=1,~\text{general } X \text{ nonsingular}$ &73\\  \hline
6&$\F(0,1,1,2)$& $\dim\Bs|-K_{\F}|=1,~\text{general } X \text{ nonsingular}$ &86 \\  \hline
7&$\F(0,1,1,3)$&$\dim\Bs|-K_{\F}|=1,~\text{general } X \text{ nonsingular}$&89 \\  \hline
8&$\F(0,1,1,4)$&$\dim\Bs|-K_{\F}|=1,~\text{general } X \text{ nonsingular}$&95\\  \hline		
9&$\F(0,0,2,2)$&$\dim\Bs|-K_{\F}|=2,~\text{general } X \text{ nonsingular}$ &91 \\ \hline
10&$\F(0,0,1,2)$& \makecell{$\dim\Bs|-K_{\F}|=2,~\text{general } X$\\ $\text{ has 3 ODP singularities along }\Bs|-K_{\F}|$} &86 \\ \hline
\end{tabular}
\label{ffoldtbl}
\vspace{0.1in}
\caption{Fourfold scrolls with general anticanonical hypersurfaces fibred in quartic $K3$ surfaces with at worst isolated singularities along the base locus}
\end{table}

\begin{proof}[Proof of Theorem~\ref{thm_quartic}] By complete symmetry of the weights $b_1=\ldots=b_4=1$, we can assume that the twists are $0=a_1\leq a_2\leq a_3\leq a_4$. 
Denoting $\F=\F(0,a_2,a_3,a_4)$ as before, we have 
\[-K_{\F} = (2-a_2-a_3-a_4)L + 4M\in \Cl(\F).\]
The equation of an anticanonical hypersurface is
\[ f(t_i, x_j) = \sum_{\q\vdash 4} \alpha_{\q} (t_i) \prod_{j=1}^4 x_j^{q_j}, 
\]
where $\q=(q_1,\ldots, q_4)$ is a partition of $4$, and $\alpha_{\q} (t_i)$ is a homogeneous polynomial of the variables $(t_1, t_2)$ of degree 
\[ \deg \alpha_{\q}  =  2 + (q_2-1)a_2+ (q_3-1)a_3+ (q_4-1)a_4,
\]
as long as this quantity is non-negative. The degree $\deg \alpha_{\q}$ increases as more of the higher indexed $x$ variables appear in a monomial $\prod_{j=1}^4 x_j^{q_j}$. 

By Proposition~\ref{toric lemma}(i), the base locus $B=\Bs|-K_{\F}|$ is defined by setting some of the Cox variables to $0$. Let us consider cases according to the dimension of $B$. If $\dim B=3$, then there is a fixed divisor in each member of the linear system, and so the general section is reducible. The cases where $B=\emptyset$ have already been classified by Mullet. The case $\dim B=0$ is not possible because of the symmetry of the $t_1, t_2$ variables: the toric base locus $B$ cannot be an isolated point. 

Assume that  $\dim B=2$. This means that $f$ cannot be divisible by any of the~$x_i$ variables. Looking at degrees of coefficients, $f$ must have 
nonzero coefficients at least for $x_3^4$ and  $x_4^4$, giving us the inequality
\[2-a_2+3a_3-a_4\geq0.\]
This also means that the base locus must be
\[ B=\{x_3=x_4=0\}\subset \F,\]
and we cannot have any terms in $f$ only involving $x_1$ and  $x_2$. This gives us the inequality 
\[2+3a_2-a_3-a_4<0.
\]
Let us now investigate potential singularities along the base locus. We have
\[
\Sing(X)\cap B = \VV\left(\frac{\partial f}{\partial x_3}\bigg|_{x_3=x_4=0},\frac{\partial f}{\partial x_4}\bigg|_{x_3=x_4=0}, x_3,x_4\right)\]
with 
\[\frac{\partial f(t_i, x_j) }{\partial x_3}\bigg|_{x_3=x_4=0}=\alpha_{3010}(t_1,t_2)x_1^3+\alpha_{2110}(t_1,t_2)x_1^2x_2+\alpha_{1210}(t_1,t_2)x_1x_2^2+\alpha_{0310}(t_1,t_2)x_2^3\]
and 
\[\frac{\partial f(t_i, x_j) }{\partial x_4}\bigg|_{x_3=x_4=0}=\alpha_{3001}(t_1,t_2)x_1^3+\alpha_{2101}(t_1,t_2)x_1^2x_2+\alpha_{1201}(t_1,t_2)x_1x_2^2+\alpha_{0301}(t_1,t_2)x_2^3.\]
To get isolated singularities along $B$, these two equations must give an isolated set of solutions on $B$. Once again recalling degrees, the condition for this is that the coefficients
$\alpha_{0310}(t_1,t_2)$ and $\alpha_{3001}(t_1,t_2)$ should be nonzero, giving us the last set of inequalities
\[2+2a_2-a_4\geq0\text{ and }2-a_2-a_3\geq0.\]
Together with $0\leq a_2\leq a_3\leq a_4$, it can be checked that the only solutions to our full set of inequalities are $\F=\F(0,0,2,2)$ and $\F=\F(0,0,1,2)$. The former already appeared in Mullet's list and indeed in this case there are no solutions to the last two equations and thus $\Sing(X)\cap B=\emptyset$. In the final case, a quick degree count
shows that for general $f$, there exist three isolated solutions. A straightforward local analysis shows that the resulting singularities on $X$ are ordinary double points, 
and $B$ is a non-Cartier Weil divisor at each singular point. It is then well known that blowing up $B$ gives a small resolution of $X$. 

The analysis in the case $\dim B=1$ is analogous; we omit the details. We do not find any new examples beyond Mullet's in this case.  
\end{proof}
\begin{remark} In families 1-2 of Table 1, the anticanonical divisor class is very ample, so by Lefschetz, the corresponding Calabi--Yau threefolds have Picard number $\rho(X)=2$. 
For families 3 and 4, the anticanonical map is a semismall morphism, so the anticanonical divisor class is {\em lef}, in the language of~\cite{dCM}. Hence by [ibid, Prop.2.1.5], 
Lefschetz still applies, and so again $\rho(X)=2$. The same was proved for family 6 in~\cite[App.A]{ruan}. It seems possible that $\rho(X)=2$ in families 1-8. However, in families 9-10 it is easy to see that the base locus $B$, or its
proper transform, give an extra divisor class in $X$ or its resolution $Y$ which cannot come from the ambient space, so the Picard number of the smooth models is at least (and 
likely to be equal to) $3$. 

One can also ask about the dimension of the space of complex deformations, the Hodge number $h^{2,1}$. For families 1-2, the dimension of the space $\M_{X\subset\F}$ of embedded deformations is 86, respectively 118, and these values agree with the corresponding Hodge numbers. Once again, this is known to be the case also for family 8, with the dimension being also 86. However, in the other cases, especially families 9-10, the determination of this Hodge number appears more difficult. 
\end{remark}

\subsection{Projective degenerations} We look at 
three of the families from Table 1 in some more detail, numbered 1,  
6 and 9: the families of anticanonical hypersurfaces in $\P^1\times\P^3$ and the scrolls
\[\F(0,1,1,2)\cong\P\left(\O_{\P^1}(-1)\oplus\O_{\P^1}\oplus\O_{\P^1}\oplus\O_{\P^1}(1)\right)\]
and \[\F(0,0,2,2)\cong\P\left(\O_{\P^1}(-1)\oplus\O_{\P^1}(-1)\oplus\O_{\P^1}(1)\oplus\O_{\P^1}(1)\right).\]
In this section, we use computer algebra, specifically Macaulay2~\cite{m2} and polymake~\cite{polymake}, to prove some of our statements. 

The first two of these families has in fact appeared some time ago in the papers of Gross~\cite{gross} and Ruan~\cite{ruan}. Gross noted that 
by considering the universal extension over the one-dimensional space 
${\rm Ext}^1(\O_{\P^1}(1), \O_{\P^1}(-1))$, adding two copies of a trivial bundle and projectivizing, one gets a deformation family 
${\mathcal F}\rightarrow\A^1$ with central fibre $\CF_0\cong \F(0,1,1,2)$ and all other fibres isomophic to $\CF_1\cong \P^1\times\P^3$. However, while general members of the anticanonical families $X_1\subset \CF_1$, $X_0\subset \CF_0$ in the two spaces
are nonsingular, as seen above, they cannot be smoothly deformed into each other via this construction. A general anticanonical section $X_1$ in the general fibre $\CF_1$ specialises in the deformation family $\CF$ to a singular Calabi--Yau threefold $\overline{X}_0$, with a curve of canonical singularities along the base locus $B\cong \P^1\subset \F(0,1,1,2)$ of the anticanonical system 
of the central fibre. It is in fact known that the general anticanonical sections $X_1, X_0$ are diffeomorphic, with Hodge numbers $(2,86)$, but not sympectic deformation equivalent~\cite[Thm.A.4.3.]{ruan}, so not algebraic deformation equivalent. The local deformation space of the singular threefold $\bar X_0$, in turn, has (at least) two components in its local deformation space, deforming to the general
$X_0$, respectively $X_1$, and thus has obstructed deformation theory~\cite[Thm.2.2]{gross}. 

We describe a projective version of this specialisation, based on natural maps defined on our scrolls in terms of bi-homogeneous coordinates. As a starting point, consider
the half-anticanonical embedding of $\CF_1=\P^1\times \P^3$ 
\[  \phi_1=\phi_{|-\frac{1}{2}K_{\P^1\times \P^3 }|}\colon  \P^1\times \P^3 \hookrightarrow \P^{19} 
\]
defined, using a slight abuse of notation, by
\[(t_1,t_2\, ; \, x_1,x_2,x_3,x_4)\mapsto (S^1(t_1, t_2) \otimes S^2(x_1,x_2,x_3,x_4)).\]
We find it convenient to use variables $z_{ijkl}$ on $\P^{19}$ with $i, j$ either $0$ or $1$, summing to $1$, and $(k,l)$ are symmetric,
with the map $\phi_1$ defined by $z_{ijkl} = t_1^i t_2^j x_kx_l$. Consider the matrix
\[ M_1=\begin{bmatrix}
	z_{1011}&z_{1012}&z_{1013}&z_{1014} &z_{0111}&z_{0112}&z_{0113}&z_{0114}\\
	z_{1012}&z_{1022}&z_{1023}&z_{1024}&z_{0112}&z_{0122}&z_{0123}&z_{0124}\\
	z_{1013}&z_{1023}&z_{1033}&z_{1034}&z_{0112}&z_{0122}&z_{0123}&z_{0124}\\
	z_{1014}&z_{1024}&z_{1034}&z_{1044}&z_{0114}&z_{0124}&z_{0134}&z_{0144}
\end{bmatrix},
\]
joined from two $4\times 4$ symmetric matrices with independent entries. 

\begin{proposition} The image of $\phi_1$ inside $\P^{19}$ is described (scheme-theoretically) by the ideal generated by $2\times 2$ minors of the matrix $M_1$:
\[\im\phi_1 \cong \VV\left(\wedge^2 M_1 \right) \subset \P^{19}.\]
\label{prop phi1} 
\end{proposition}
\begin{proof} As in the standard Segre embedding, on the image of $\phi_1$ the rows, respectively columns, of $M_1$ are proportional to each other, so the rank of $M_1$
is at most $1$ on the image. So $\im\phi_1 \subset \VV\left(\wedge^2 M_1 \right)$. From here, the easiest way to conclude is to check by Macaulay2 that the quadrics in $\wedge^2 M_1$ define a $4$-dimensional irreducible variety in $\P^{19}$, so this inclusion is in fact an equality. \end{proof}

Next, consider the half-anticanonical map 
\[  \phi_0=\phi_{|-\frac{1}{2}K_\F|}\colon \F(0,1,1,2) \dashrightarrow  \P^{19}
\]
on the scroll $\CF_0\cong \F(0,1,1,2)$, defined in Cox coordinates $(t_1, t_2; x_1,x_2,x_3,x_4)$ by
\[
 (t_i; x_j)\mapsto (x_1x_2,x_1x_3,S^1(t_i)x_1x_4,S^1(t_i)\otimes S^2(x_2,x_3),S^2(t_i)x_2x_4,S^2(t_i)x_3x_4,S^3(t_i)x_4^2).
\]
On the target $\P^{19}$, we are going to use coordinates $y_{ijmn}$ with the correct index conventions so that the map $\phi_0$ be given by 
$y_{ijmn} = t_1^i t_2^j x_m x_n$. 

The variety $\F(0,1,1,2)$ is not Fano, the rational map $\phi_0$ has indeterminacy locus given by the rational curve $B=\{x_2=x_3=x_4=0\}\subset\F(0,1,1,2)$, the base locus of its (half) anticanonical system. Denote by $Q=\overline{\im \phi_0}\subset\P^{19}$ the closure of the image of $\phi_0$. We give descriptions of $Q$ both as an embedded projective variety, and as a toric variety. First, consider the matrix 
$$M_0=\begin{bmatrix}
	y_{0012}&y_{1022}&y_{1023}&y_{0122}&y_{0123}&y_{2024}&y_{1124}&y_{0224}\\
	y_{0013}&y_{1023}&y_{1033}&y_{0123}&y_{0133}&y_{2034}&y_{1134}&y_{0234}\\
	y_{1014}&y_{2024}&y_{2034}&y_{1124}&y_{1134}&y_{3044}&y_{2144}&y_{1244}\\
	y_{0114}&y_{1124}&y_{1134}&y_{0224}&y_{0234}&y_{2144}&y_{1244}&y_{0234}
\end{bmatrix}.$$
\begin{proposition}\label{prop phi0} \begin{enumerate}
\item The closure $Q=\overline{\im \phi_0}\subset\P^{19}$ is described (scheme-theoretically) by the ideal generated by $2\times 2$ minors of the matrix $M_0$:
\[ Q=\overline{\im \phi_0}\cong \VV\left(\wedge^2 M_0 \right) \subset \P^{19}.\]
\item There is a distinguished divisor $\P^3\cong D\subset Q\subset\P^{19}$ embedded in $\P^{19}$ as a projective linear subspace, defined by the condition that all variables except those in the first column of $M_0$ vanish. 
\item The Hilbert polynomials of $Q\subset \P^{19}$ and $\F(0,1,1,2)\subset \P^{19}$ agree. \end{enumerate}
\end{proposition}
\begin{proof} (1) can be proved by an argument identical to that in the previous Proposition. On the image of $\phi_0$ the rows, respectively columns, of $M_0$ are proportional to each other. So $\im\phi_0 \subset \VV\left(\wedge^2 M_0 \right)$. By Macaulay2, the quadrics in $\wedge^2 M_0$ define a $4$-dimensional irreducible variety in $\P^{19}$, so we deduce the first statement. (2) is immediate. (3) can be checked by Macaulay2. \end{proof}

Although not strictly necessary for what follows, we also give a description of the projective variety $Q$ as a toric variety, to make contact with the discussion in~\cite[App.A]{ruan}.
First, recall the toric description of $\CF_0\cong \F(0,1,1,2)$ from Section~\ref{sect_toric}: we have one-dimensional rays $\sigma_1,\sigma_2, \rho_1, \ldots, \rho_4\in N$ in a rank-4
lattice $N$ based by $\sigma_2, \rho_2, \rho_3, \rho_4$, with $\rho_1=-\rho_2-\rho_3-\rho_4$ and $\sigma_1=-\sigma_2+\rho_2+\rho_3+2\rho_4$. The fan $\Sigma$ with eight maximal dimensional cones
\[\tau_1=(\sigma_1, \rho_1, \rho_2, \rho_3), \tau_2=(\sigma_1, \rho_1, \rho_2, \rho_4), \tau_3=(\sigma_1, \rho_1, \rho_3, \rho_4), \tau_4=(\sigma_1, \rho_2, \rho_3, \rho_4)\]
and
\[\tau_5=(\sigma_2, \rho_1, \rho_2, \rho_3), \tau_6=(\sigma_2, \rho_1, \rho_2, \rho_4), \tau_7=(\sigma_2, \rho_1, \rho_3, \rho_4), \tau_8=(\sigma_2, \rho_2, \rho_3, \rho_4)\]
defines a toric variety $X_{N,\Sigma}$ isomorphic to $\F(0,1,1,2)$. 

Define a new ray $\rho_5=-\rho_1 = \rho_2+\rho_3+\rho_4$. Consider the cones $\tau_1, \tau_2, \tau_3, \tau_5, \tau_6, \tau_7$, as well as new cones
\[ \tau_9= (\sigma_2, \rho_2, \rho_3, \rho_5),\tau_{10}=(\sigma_1, \rho_2, \rho_3, \rho_5), \tau_{11}= (\sigma_1,\sigma_2, \rho_2, \rho_4, \rho_5), \tau_{12}= (\sigma_1,\sigma_2, \rho_3, \rho_4, \rho_5). 
\]
It can be checked that these $10$ cones also give a fan $\Sigma'$ for the same lattice $N$. 

\begin{proposition} The projective variety $Q$ is isomorphic to the toric variety $X_{N,\Sigma'}$. The toric divisor corresponding to the new ray $\rho_5$ is the distinguished divisor
$\P^3\cong D\subset Q$. 
\end{proposition}
\begin{proof} By standard toric geometry, the sections of the half-canonical linear system on  $X_{N,\Sigma}\cong\F(0,1,1,2)$ are the lattice points contained in a certain
rational polytope $P_1\subset M_\R$, where $M$ is the dual lattice of $N$. This polytope can be computed explicitly by polymake: it is a rational, 
non-integral polytope with $20$ lattice points as expected. The convex hull $P_2$ of the lattice points in $P_1$ is a normal lattice polytope; the corresponding projective toric variety is thus our variety $Q\subset\P^{19}$. The dual fan, computed also by polymake, has $10$ maximal cones, and can be identified with the fan $\Sigma'$ described above. A check through the construction shows that the distinguished divisor $\P^3\cong D\subset Q$ is toric; it is then easy to see that it must be the divisor corresponding to the ray $\rho_5$. 
\end{proof}

\begin{remark} Repeating the above analysis for the anticanonical map of $\F=\F(0,1,1,2)$, we obtain the polytope $2\cdot P_1\subset M_\R$, which is now integral, as well as normal. Its dual fan is precisely the fan given by $9$ maximal cones described by~\cite[App. A]{ruan}; this gives the full anticanonical model 
\[ \tilde Q = {\rm Proj} \left(\oplus_{m\geq 0} H^0(\F, -m K_{\F})\right)\]
 of $\F(0,1,1,2)$. The polytope $2\cdot P_1$ has one extra lattice point compared to the polytope $2\cdot P_2$. The algebraic interpretation of these facts is that the space $H^0(\F, - K_{\F})$ needs one extra generator compared to products of elements of $H^0(\F, -\frac{1}{2} K_{\F})$, and generates the full anticanonical algebra. In particular, there exists an embedding of the anticanonical model
 \[ \tilde Q \subset\P^{20}[1^{19}, 2]\]
 into a weighted projective space. The birational map $\tilde Q \dashrightarrow Q$ is given by projection from the point $(0:\ldots:0:1)\in \P^{20}[1^{19}, 2]$, and has indeterminacy
 locus $D\subset Q$ in its target. The process that creates $\tilde Q$ from $(Q,D)$ is similar to the unprojection construction~\cite{reid_ki} studied by Reid and students. 
\end{remark}

Return to our projective models in $\P^{19}$. Our main result in this section is the description of an explicit family in $\P^{19}$ which exhibits a degeneration of its general fibre, isomorphic to $\CF_1\cong \P^1\times \P^3$, to the half-anticanonical model $Q$ of $\CF_0\cong \F(0,1,1,2)$. We also look at how anticanonical hypersurfaces specialise in the family. 
Consider the $4$ by $9$ matrix

\[ \begin{array}{rcl} N & =&  \left[ \begin{array}{ccccc}
	y_{0012}& y_{1022}&y_{1023}&y_{2024}&y_{1124}+ ty_{0012} \\
	y_{0013}&y_{1023}&y_{1033}&y_{2034}&y_{1134}+ty_{0013} \\
	y_{1014}&y_{2024}&y_{2034}&y_{3044}&y_{2144}+ty_{1014} \\
	y_{0114}&y_{1124}+ty_{0012}&y_{1134}+ty_{0013}&y_{2144}+ty_{1014}&y_{1244}+ty_{0114} 
 \end{array}\right. \\ \\
 & & \left. \begin{array}{cccc}
 y_{0122}  &y_{0123}&y_{1124}-ty_{0012}&y_{0224}\\
 y_{0123}&y_{0133}&y_{1134}-ty_{0013}&z_{0234}\\ 
 y_{1124}-ty_{0012}&y_{1134}-ty_{0013}&y_{2144}-ty_{1014}&y_{1244}-ty_{0114}\\
 y_{0224}&z_{0234}&y_{1244}-ty_{1014}&y_{0344}\end{array}\right].
\end{array}
\]

\begin{theorem}
\begin{enumerate}
\item Define the variety $\CQ\subset \P^{19}\times\A^1$ over $\A^1$ by the equations
\[ \CQ =  \VV\left(\wedge^2 N \right) \subset \P^{19}\times\A^1.
\]
Then the natural map $\CQ\to\A^1$ is a flat family of projective varieties, with central fibre $\CQ_0\cong Q\subset\P^{19}$, and all other fibres isomorphic to $\CQ_1\cong \P^1\times\P^3\subset\P^{19}$.
\item An anticanonical hypersurface  $X_1\subset\CQ_1\cong \P^1\times\P^3$ specialises in the family $\CQ$ to a {\em reducible} threefold $\underline{X}_0\subset Q$, which contains a double copy of the distinguished toric divisor $\P^3\subset Q$. 
\end{enumerate}
\end{theorem}
\begin{proof} For (1), note that setting $t=0$ in $N$, we recover the matrix $M_0$ from Proposition~\ref{prop phi0} (with a repeated column), so indeed
the central fibre of the family is isomorphic to $Q$.  On the other hand, for $t\neq 0$, the first column of $N|_t$ is superfluous, as it is a linear combination of other columns. The rest of the matrix has the structure of a join of two $4\times 4$ square matrices with independent entries. An obvious linear change of variables brings it into the form of the $M_1$ matrix from Proposition~\ref{prop phi1}. Thus indeed all nonzero fibres are isomorphic to $\P^1\times\P^3$. All fibres are projective with the same Hilbert polynomial, so the family is flat. 

For (2), a detailed check shows that if one starts with any quadric monomial in the $z_{ijkl}$ variables, performing the linear change of variables of the previous paragraph and setting $t=0$ gives a quadric in the $y_{ijkl}$ variables that vanishes on the distinguished $\P^3\subset Q$. This means that for any degeneration to $t=0$ of a quadric section of $\CQ_1\cong \P^1\times\P^3$ in $\P^{19}$, in other words any degeneration of an anticanonical hypersurface in $\P^1\times\P^3$ in this family, the flat limit $\underline{X}_0\subset \CQ_0\cong Q$ at $t=0$ contains the distinguished $\P^3\subset Q\subset\P^{19}$ as a component. A Macaulay2 calculation shows that the multiplicity along this component is $2$. 
\end{proof}
This is our analogue of Gross' theorem about degenerations of anticanonical hypersurfaces in the family: in the projective picture, the smooth Calabi-Yau $X_1$ specialises 
to a reducible threefold $\underline{X}_0$, while the most general anticanonical section in $Q$ is irreducible.

We briefly comment on one further family from the list above. Consider 
\[\F(0,0,2,2) \cong\P\left(\O_{\P^1}(-1)\oplus\O_{\P^1}(-1)\oplus\O_{\P^1}(1)\oplus\O_{\P^1}(1)\right).\]
 The underlying vector bundle is a further degeneration of $\O_{\P^1}(-1)\oplus\O_{\P^1}\oplus\O_{\P^1}\oplus\O_{\P^1}(1)$, so this scroll is a further degeneration of $\F(0,1,1,2)$. 
 However, looking at the table, the dimension of $\M_{X\subset\F}$ grows, proving that the anticanonical hypersurface in $\F(0,1,1,2)$ undergoes a more drastic degeneration,
 with a smoothing into $\F(0,0,2,2)$ with more interesting vanishing cycles that contribute to $h^{1,2}$. In our projective picture, it can be checked that the half-anticanonical model 
 of $\F(0,0,2,2)$ is still mapping to~$\P^{19}$. However, the equations of the image are more complicated. There is an explicit degeneration similar to the one given by the matrix $N$ above; but the flat limit gives a reducible variety, one of whose components is the anticanonical image. For details, see~\cite{mboya}. 


\subsection{Other K3 families}
\label{sec:deg2}

In the list of Iano-Fletcher--Reid, the next simplest entries are the surfaces $S_5\subset\P(1^3, 2)$, $S_6\subset\P(1^3, 3)$ and $S_6\subset\P(1^2, 2^2)$. The second of these
is nonsingular, and is recognisable as the family of degree $2$ surfaces obtained as double covers of $\P^2$ branched over a sextic plane curve; the other two have one, respectively three, $A_1$ singularities. Following the same method as above, we can extend Mullet's list by the following examples; for detailed proofs, see~\cite{mboya}. 

\begin{theorem} Families of anticanonical hypersurfaces $X\in |-K_{\F}|$ in weighted scrolls, fibred in quintic or sextic $K3$ surfaces, containing  isolated singularities along the base locus $B=\Bs|-K_{\F}|$ and quasi-smooth outside $B$, are listed in Table 2. 
In all cases, the base locus $B$ is a surface, and blowing up $B$ in $X$ as well as resolving the remaining quotient singularities gives a smooth projective Calabi–Yau model.
\label{thm_others}
\end{theorem}

\begin{table}[h]
\centering
\begin{tabular}{|c|c|c|c| }\hline
Weighted scroll $\F$ & Description of general $\text{}X\in|-K_{\F}|$ \\ \hhline{|=|=|}
$\F(0,0,1,2\,|\,1^3,3)$& \makecell{General $X\in|-K_{\F}|$ has 5 isolated ODP singularities on $\Bs(|-K_{\F}|)$\\ and three $\frac{1}{3}(1,1,1)$ singularities} \\ \hline
$\F(0,0,2,1\,|\,1^3,3)$& \makecell{General $X\in|-K_{\F}|$ has 3 isolated ODP singularities on $\Bs(|-K_{\F}|)$ \\ and one $\frac{1}{3}(1,1,1)$ singularity}  \\ \hline
$\F(0,0,1,2\,|\,1^2,2^2)$& \makecell{General $X\in|-K_{\F}|$ has 4 isolated ODP singularities on $\Bs(|-K_{\F}|)$\\ and a smooth curve of $A_1$ singularities} \\ \hline
$\F(0,2,0,1\,|\,1^2,2^2)$& \makecell{General $X\in|-K_{\F}|$ has 2 isolated ODP singularities on $\Bs(|-K_{\F}|)$\\ and two disjoint smooth curves of $A_1$ singularities} \\ \hline

$\F(0,0,1,2\,|\,1^3,2)$& \makecell{General $X\in|-K_{\F}|$ has 4 isolated ODP singularities on $\Bs(|-K_{\F}|)$\\ and one smooth curve of $A_1$ singularities} \\ \hline
$\F(0,1,2,0\,|\,1^3,2)$& \makecell{General $X\in|-K_{\F}|$ has 2 isolated ODP singularities on $\Bs(|-K_{\F}|)$\\ and one smooth curve of $A_1$ singularities} \\ \hline
$\F(0,0,2,1\,|\,1^3,2)$& \makecell{General $X\in|-K_{\F}|$ has 3 isolated ODP singularities on $\Bs(|-K_{\F}|)$\\ and one smooth curve of $A_1$ singularities} \\ \hline

\end{tabular}
\vspace{0.1in}
\caption{Weighted fourfold scrolls with general anticanonical hypersurfaces fibred in quintic or sextic $K3$ surfaces with at worst isolated singularities along the base locus}
\label{weighted table}
\end{table} 

\section{General considerations}\label{sect_end}

When looking for examples of projective varieties, it is natural to look in easily accessible ambient spaces. Examples of projective fibrations similarly arise in naturally fibered 
ambient spaces, such as scrolls or more generally projectivized bundles; from many possible examples, we mention~\cite{BCZ}, where examples of threefold Mori fibre spaces are 
constructed as hypersurfaces in weighted scrolls. It seems worthwhile also to attempt to go the other way, repeating the general discussion of the Introduction in a relative context. 

Let $f\colon X\to B$ be a fibration of normal projective varieties, with $f_*\O_X\cong \O_B$ so that $f$ has connected fibres. Let $D$ an ample $\Q$-Cartier Weil divisor on the base $X$, and $H$ an ample $\Q$-Cartier Weil divisor on the total space $X$. Then we get a bigraded algebra
\[ R_X = \bigoplus_{m,n\in\Z} H^0(X, \O_X(mf^*H + nD)), \]
containing the graded subalgebra
\[ R_B = \bigoplus_{m\geq 0}^\infty H^0(B, \O_B(mH)). \]
The discussion of the Introduction applies to the base $B$: choosing generators $r_1, \ldots, r_k$ of positive weights $c_i$ of $R_B$, we get a surjection
\[ S_B=\C[t_1, \ldots, t_k] \twoheadrightarrow R_B\] of graded algebras from a free graded algebra, and thus a closed inclusion 
\[ (B,D) \hookrightarrow \left(\P^{k-1}[c_1, \ldots, c_k], \O(1)\right). \]
Suppose that $R_X$ is a {\em finitely generated} algebra over its subalgebra $R_B$, a kind of ``Mori dream space'' assumption. Choosing a set $r_1, \ldots, r_n$ of generators of $R_X$ of bidegree $(-a_j, b_j)$ over $R_B$, we then get
a diagram 
$$\begin{tikzcd}
S_X \arrow[r, two heads] & R_X \\
S_B\arrow[u,"{}"]\arrow[r, two heads] & R_B\arrow[u,"{}"]
\end{tikzcd}$$
of compatible surjections of bigraded algebras, where 
\[S_X = \C[t_1, \ldots, t_k, x_1, \ldots, x_n]
\]
is a free bigraded algebra with generators of bidegrees $(c_i, 0)$, respectively $(-a_j, b_j)$. 
On the geometric side, this diagram gives the diagram of projective varieties
$$\begin{tikzcd}
X\arrow[d,"f"]\arrow[r, hook] &\F\arrow[d, "\pi"]\\
B\arrow[r,hook]&\P^{k-1}[c_i],
\end{tikzcd}$$
embedding our original fibration $f\colon X\to B$ into a general weighted scroll 
\[\pi\colon \F=\F(c_1, \ldots c_k ||\, a_1,\ldots, a_n |\, b_1, \ldots, b_n)\rightarrow\P^{k-1}[c_i], \]
introduced in Remark~\ref{general}. Note that this in particular embeds all fibres $X_b=f^{-1}(b)$ into the weighted projective fibres $\P^{n-1}[b_j]$ of $\pi$. One could then start a programme
of classifying and studying cases by increasing codimension, as in the absolute case. 

The situation however is not so simple: the requirements of finite generation of the algebra~$R_X$, and a surjective map $S_X\rightarrow R_X$ from a free bigraded algebra, are too strong. 
Even in our basic hypersurface examples $X\subset\F$, the natural restriction map
\[ \bigoplus_{m,n\in\Z} H^0(\F, \O_\F(mL  + n M)) \rightarrow \bigoplus_{m,n\in\Z} H^0(X, \O_X(mL|_X  + n M|_X))
\]
may fail to be surjective. In fact by~\cite[Thm.3.1]{AL}, this map is definitely {\em not} surjective for families 1-2 of Table 1, and we do not know whether the algebra 
\[R_X= \bigoplus_{m,n\in\Z} H^0(X, \O_X(mL|_X  + n M|_X))\]
is finitely generated in these examples. (Note that, curiously, this map is surjective for at least one member of family 10 by~\cite[Rem.2.5]{AL}.) What appears to be needed in general then
is a way to capture {\em enough} of the bigraded algebra $R_X$ via a free bigraded algebra 
to be able to describe the fibration $f\colon X\to B$, at least in favourable cases, as embedded in a weighted scroll. An alternative, explored in the recent 
preprint~\cite{CP} by Coughlan and Pignatelli in particular, is to study conditions under which $X$ embeds into a relative weighted projective bundle over the base~$B$ using pushforwards of $\O_X(nD)$ to the base, following on from the local (on the base) analysis of Reid~\cite{reid_sm}. 
Finding general conditions under which we get embeddings of low codimension appears to us to be a question worthy of further study.

\end{document}